\documentclass{elsart3-1}

\usepackage{amsmath,amssymb,mathtools,xspace}
\usepackage{enumerate}

\usepackage[english,francais]{babel}
\usepackage{color}

\newtheorem{theorem}{Theorem}
\newtheorem{lemma}[theorem]{Lemma}

\newtheorem{remark}{\it Remark\/}
\newenvironment{proof}{\paragraph*{\textbf{Proof}}}{\hfill$\square$}

\usepackage{def}

\journal{Applied Mathematics Letters}
\begin{document}

\begin{frontmatter}

\selectlanguage{english}
\vspace{-2cm}

\title{A converse to Fortin's Lemma in Banach spaces}

\selectlanguage{english}
\author[authorlabel1]{Alexandre Ern}
\ead{ern@cermics.enpc.fr} \and
\author[authorlabel2]{Jean-Luc Guermond}
\ead{guermond@math.tamu.edu}

\address[authorlabel1]{Universit\'e Paris-Est, CERMICS (ENPC), 77455 Marne la Vall\'ee Cedex 2, France}
\address[authorlabel2]{Department of Mathematics, Texas
  A\&M University 3368 TAMU, College Station, TX 77843, USA}

\begin{abstract}
  \selectlanguage{english} The converse of Fortin's Lemma in Banach
  spaces is established in this Note.

\vskip 0.5\baselineskip

\vskip 0.5\baselineskip \noindent
{

}

\end{abstract}


\end{frontmatter}

\selectlanguage{english}

\section{Introduction}

Let $V$ and $W$ be two complex Banach spaces equipped with the norms
$\|\SCAL\|_V$ and $\|\SCAL\|_W$, respectively.  We adopt the
convention that dual spaces are denoted with primes and are composed
of antilinear forms; complex conjugates are denoted by an
overline. Let $a$ be a sesquilinear form on $V\times W$ (linear
w.r.t.~its first argument and antilinear w.r.t.~its second
argument). We assume that $a$ is bounded, \ie
\begin{equation} \label{eq:a_bnd}
\|a\| := \sup_{v\in V}\sup_{w\in W} \frac{|a(v,w)|}{\|v\|_V\|w\|_W} < \infty,
\end{equation}
and that the following inf-sup condition holds:
\begin{equation} \label{eq:a_infsup}
\alpha := \inf_{v\in V}\sup_{w\in W} \frac{|a(v,w)|}{\|v\|_V\|w\|_W} > 0.
\end{equation}
Here and in what follows, arguments in infima and suprema are implicitly assumed to be  nonzero.

Let $V_h\subset V$ and $W_h\subset W$ be two finite-dimensional subspaces equipped with the norms of $V$ and $W$, respectively. A question of fundamental importance is to assert the following discrete inf-sup condition:
\begin{equation} \label{eq:a_infsup_disc}
\hat \alpha := \inf_{v_h\in V_h}\sup_{w_h\in W_h} \frac{|a(v_h,w_h)|}{\|v_h\|_V\|w_h\|_W} > 0.
\end{equation}
The aim of this Note is to prove the following result.

\smallskip

\begin{theorem}[Fortin's Lemma with converse] \label{th:fortin} Under
  the above assumptions, consider the following two statements: \berom
\item \label{item1:lemfortin} There exists a map
$\Pi_h: W\to W_h$ and a real number $\gamma_\Pi>0$ such that
$a(v_h,\Pi_hw-w)=0$, for all $(v_h,w)\in V_h\times W$, and
$\gamma_\Pi \|\Pi_hw\|_W\le \|w\|_W$ for all $w\in W$.
\item \label{item2:lemfortin} The discrete inf-sup condition
  \eqref{eq:a_infsup_disc} holds.  \eerom Then,
  \itemref{item1:lemfortin} $\Rightarrow$ \itemref{item2:lemfortin}
  with $\hat\alpha= \gamma_\Pi \alpha$.  Conversely,
  \itemref{item2:lemfortin} $\Rightarrow$ \itemref{item1:lemfortin}
  with $\gamma_\Pi = \frac{\hat\alpha}{\|a\|}$, and $\Pi_h$ can be
  constructed to be idempotent. Moreover, $\Pi_h$ can be made
  linear if $W$ is a Hilbert space.
\end{theorem}

The statement (i) $\Rightarrow$ (ii) in Theorem~\ref{th:fortin} is classical and is
known in the literature as Fortin's Lemma, see \cite{Fo77}
and~\cite[Prop.~5.4.3]{BoBrF:13}. 
It provides an effective
tool to prove the discrete inf-sup condition~\eqref{eq:a_infsup_disc}
by constructing explicitly a Fortin operator $\Pi_h$. 
We briefly outline a proof that (i) $\Rightarrow$ (ii)
for completeness. Assuming \itemref{item1:lemfortin}, we have
\begin{align*}
\sup_{w_h\in W_h} \frac{|a(v_h,w_h)|}{\|w_h\|_W}
&\geq  
\sup_{w\in W} \frac{|a(v_h,\Pi_hw)|}{\|\Pi_hw\|_W}
= 
\sup_{w\in W} \frac{|a(v_h,w)|}{\|\Pi_hw\|_W} 
\geq \gamma_\Pi
\sup_{w\in W} \frac{|a(v_h,w)|}{\|w\|_W} \geq \gamma_\Pi\alpha\, 
\|v_h\|_V,
\end{align*}
since $a$ satisfies~\eqref{eq:a_infsup} and $V_h\subset V$.
This proves~(ii) with $\hat\alpha=\gamma_\Pi\alpha$.

The converse
(ii) $\Rightarrow$ (i) is
of independent theoretical interest and is the main object of this
Note. This property is useful when it is easier to prove the
discrete inf-sup condition directly rather than constructing a Fortin
operator. Another application of current interest 
is the analysis framework for
discontinuous Petrov--Galerkin methods (dPG) recently proposed
in~\cite{CarDemGop:14} which includes the existence of a Fortin
operator among its key assumptions.
Incidentally, we observe that there is a gap in the stability constant
$\gamma_\Pi$ between the direct and converse statements, since the
ratio of the two is equal to $\frac{\|a\|}{\alpha}$ (which is
independent of the spaces $V_h$ and $W_h$).

\section{Proof of Theorem~\ref{th:fortin}}

Assume that the discrete inf-sup condition~\eqref{eq:a_infsup_disc} holds.
Let $A_h:V_h\rightarrow W_h'$ be the
operator defined by
$\langle A_hv_h,w_h\rangle_{W_h',W_h} := a(v_h,w_h)$.  
Identifying $V_h''$ with $V_h$
and $W_h''$ with $W_h$ (since these spaces are finite-dimensional), 
we consider the linear map $A_h\adj:W_h\to V_h'$. Our goal is to construct
a right-inverse map $R_{A_h\adj} : V_h'\to W_h$ (possibly nonlinear) such that, 
for all $\theta_h\in V_h'$, $A_h\adj (R_{A_h\adj}(\theta_h))
= \theta_h$ and $\hat\alpha \|R_{A_h\adj}(\theta_h)\|_W \le
\|\theta_h\|_{V_h'}$. Indeed, if such a map exists, we can consider
the linear map $\Theta : W \to V_h'$ such that,
for all $w\in W$, $\langle \Theta(w),v_h\rangle_{V_h',V_h} := \overline{a(v_h,w)}$ for all
$v_h\in V_h$. 
Then defining $\Pi_h = R_{A_h\adj}\circ \Theta : W \to W_h$ yields 
\begin{align*}
a(v_h,\Pi_h(w)) &= \langle A_h v_h,R_{A_h\adj}(\Theta(w))\rangle_{W_h',W_h} 
= \overline{\langle A_h\adj (R_{A_h\adj}(\Theta(w))),v_h\rangle_{V_h',V_h}}  \\
&= \overline{\langle \Theta(w), v_h\rangle_{V_h',V_h}} 
=  a(v_h,w),
\end{align*}
which proves that $a(v_h,\Pi_h(w)-w)=0$ for all $w\in W$. Moreover,
\begin{align*}
 \hat\alpha\|\Pi_h(w)\|_W=\hat\alpha \|R_{A_h\adj}(\Theta(w))\|_W 
\le \|\Theta(w)\|_{V_h'}
\le \|a\| \|w\|_{W},
\end{align*}
which proves that $\frac{\hat\alpha}{\|a\|} \|\Pi_h(w)\|_W
\le \|w\|_W$. In addition, we observe that
\[
\langle \Theta (R_{A_h\adj}(\theta_h)),v_h\rangle_{V_h',V_h} = \overline{\langle A_hv_h,R_{A_h\adj}(\theta_h)\rangle_{W_h',W_h}} = \langle A_h\adj(R_{A_h\adj}(\theta_h)),v_h\rangle_{V_h',V_h} = \langle \theta_h,v_h\rangle_{V_h',V_h},
\]
for all $v_h\in V_h$, which proves that $\Theta (R_{A_h\adj}(\theta_h))=\theta_h$ for all $\theta_h\in V_h'$. As a result, $\Pi_h(\Pi_h(w)) = R_{A_h\adj}(\Theta\circ R_{A_h\adj}(\Theta(w))) = R_{A_h\adj}(\Theta(w)) = \Pi_h(w)$, \ie $\Pi_h$ is idempotent. 

It remains to build the right-inverse map $R_{A_h\adj}$ to complete
the proof.  We can rewrite~\eqref{eq:a_infsup_disc} as follows:
\[
\hat \alpha = \inf_{v_h\in V_h}\sup_{w_h\in W_h} \frac{|\langle A_hv_h,w_h\rangle_{W_h',W_h}|}{\|v_h\|_V\|w_h\|_W} > 0.
\]
Let us assume first that $W$ is a Hilbert space. Let $K_h$ be the
orthogonal complement of $\KER(A_h\adj)$ in $W_h$, \ie $W_h =K_h\oplus \KER(A_h\adj)$. Observing
  that $A_h\adj : K_h \to V_h'$ is bijective, we set $R_{A_h\adj} =
  (A_{h|K_h}\adj)^{-1} : V_h'\to K_h\subset W_h$. Then $A_h\adj
  R_{A_h\adj}\theta_h =\theta_h$ for all $\theta_h\in V_h'$, by
  definition, and 
\begin{align*}
\hat\alpha & = \inf_{v_h\in V_h} \sup_{w_h\in W_h} \frac{|\langle A_hv_h,w_h\rangle_{W_h',W_h}|}{\|v_h\|_V\|w_h\|_W} 
=  \inf_{w_h\in W_h} \sup_{v_h\in V_h} \frac{|\langle A_hv_h,w_h\rangle_{W_h',W_h}|}{\|v_h\|_V\|w_h\|_W} \\
& \le \inf_{w_h\in K_h}\sup_{v_h\in V_h} \frac{|\langle A_hv_h,w_h\rangle_{W_h',W_h}|}{\|v_h\|_V\|w_h\|_W}
=\inf_{\theta_h\in V_h'}\sup_{v_h\in V_h}\frac{|\langle A_hv_h,R_{A_h\adj}\theta_h\rangle_{W_h',W_h}|}{\|v_h\|_V\|R_{A_h\adj}\theta_h\|_W}\\
&=\inf_{\theta_h\in V_h'}\sup_{v_h\in V_h}\frac{|\langle A_h\adj R_{A_h\adj}\theta_h,v_h\rangle_{V_h',V_h}|}{\|v_h\|_V\|R_{A_h\adj}\theta_h\|_W} 
= \inf_{\theta_h\in V_h'} \frac{\|\theta_h\|_{V_h'}}{\|R_{A_h\adj}\theta_h\|_W},
\end{align*}
where the first equality results from Lemma~\ref{lem:infsup_infsup}
below (this lemma provides an abstract counterpart of the fact that
the singular values of a square matrix and its transpose coincide;
this algebraic result could be invoked here directly).  This implies
that $\hat\alpha \|R_{A_h\adj}\|_{W} \le \|\theta_h\|_{V_h'}$ for all
$\theta_h\in V_h'$.
Hence $R_{A_h\adj}$ has the desired properties; note that $R_{A_h\adj}$ is linear.

In the more general setting of Banach spaces, we
set $Y:=W_h$, $Z:=V_h'$, and $B:=A_h\adj$. Identifying 
$A_h\adjadj$ with $A_h$, we obtain
\[
\hat\alpha=\inf_{z'\in Z'} \sup_{y\in Y} \frac{|\langle B\adj z', y\rangle_{Y',Y}|}{\|z'\|_{Z'} \|y\|_Y}>0.
\]
We now apply Lemma~\ref{Lem:sur2infsup} below and infer that 
there exists a right-inverse map $R_{A_h\adj} : V_h'\to W_h$ with the desired properties.

\smallskip

\begin{remark}[Linearity and uniform stability] Assume that we have at hand a
    sequence of finite-dimensional subspaces $\{V_h\}_{h\in \calH}$,
    $\{W_h\}_{h\in \calH}$.  Assume the
    existence of a decomposition $W_h=\KER(A_h\adj) \oplus K_h$ that
    is uniformly stable with respect to $h\in\calH$, \ie there is
    $\kappa>0$, independent of $h\in\calH$, such that the induced
    projector $\pi_{K_h} : W_h\to K_h$ satisfies $\kappa
    \|\pi_{K_h}w_h\|_{W}\le \|w_h\|_W$ for all $w_h\in W_h$. This
    property holds in the Hilbertian setting with $\kappa=1$. Then,
    even for Banach spaces,
    one can use the reasoning above for Hilbert spaces to build
    a linear operator $\Pi_h$ that is uniformly bounded; the only
    difference is the bound $\|R_{A_h\adj}(\theta_h)\|_{W} \le
    (\kappa\hat\alpha)^{-1}\|\theta_h\|_{V_h'}$ leading to $\gamma_\Pi
    = \frac{\kappa\hat\alpha}{\|a\|}$.
\end{remark}

\section{Operators in Banach spaces}

Let $Y$ and $Z$ be two complex Banach spaces equipped with the norms
$\|\SCAL\|_Y$ and $\|\SCAL\|_Z$, respectively. Let $B : Y \to Z$ be a
bounded linear map.

\smallskip

\begin{lemma}[Inf-sup] \label{lem:infsup_infsup}
Assume that $B$ is bijective and that $Y$ is reflexive. Then
\begin{equation}
\inf_{y \in Y} \sup_{z'\in Z'} \frac{|\langle z',By\rangle_{Z',Z}|}{\|z'\|_{Z'} \|y\|_Y} 
=
\inf_{z'\in Z'} \sup_{y\in Y} \frac{|\langle z',By\rangle_{Z',Z}|}{\|z'\|_{Z'} \|y\|_Y}.
\label{inf_sup_equal_inf_sup}
\end{equation}
\end{lemma}

\begin{proof}
Denote by $l$ and $r$ the left- and right-hand side of
\eqref{inf_sup_equal_inf_sup}, respectively.  The left-hand side being
equal to $l$ means that $l$ is the largest number such that $\|B
y\|_{Z} \ge l\, \|y\|_{Y}$ for all $y$ in $Y$.  Let $z'\in Z'$ and $z\in
Z$.  Since $B$ is surjective, there is $y_z\in Y$ so that $B y_z=z$
and the previous statement regarding $l$ implies that $l\, \|y_z\|_{Y} \le
\|z\|_Z$. This in turn implies that
\begin{align*}
\|z'\|_{Z'} &= \sup_{z\in Z} \frac{|\langle z', z\rangle_{Z',Z}|}{\|w\|_{Z}} 
= \sup_{z\in Z} \frac{|\langle w', B y_z\rangle_{Z',Z}|}{\|z\|_{Z}} 
= \sup_{z\in Z} \frac{|\langle B\adj w', y_z\rangle_{Y',Y}|}{\|z\|_{Z}} \\
&\le \|B\adj z'\|_{Y'} \sup_{z\in Z} \frac{\|y_z\|_Y}{\|z\|_{Z}} 
\le  \frac{1}{l} \|B\adj z'\|_{Y'},
\end{align*}
which implies $l\le r$. That $r\le l$ is proved similarly by working
with $Z'$ in lieu of $Y$, $Y'$ in lieu of $Z$ and $B\adj$ in lieu of $B$ (which is also surjective). We infer that
\[
\inf_{z' \in Z'} \sup_{y''\in Y''} \frac{\langle y'' , B\adj z'\rangle_{Y'',Y'}}{\|y''\|_{Y''} \|z'\|_{Z'}} 
\le
\inf_{y''\in Y''} \sup_{z'\in Z'} \frac{\langle y'' , B\adj z'\rangle_{Y'',Y'}}{\|y''\|_{Y''} \|z'\|_{Z'}},
\]
and we conclude using the reflexivity of $Y$.
\end{proof}

\smallskip

The following result is a consequence of Banach's Open Mapping and Closed Range Theorems, see, \eg~\cite[Lem.~A.36 \& A.40]{ErnGuermond_FEM}.

\smallskip

\begin{lemma}[Surjectivity] \label{lem:surj}
The following three statements are equivalent:
\berom
\item $B: Y \rightarrow Z$ is surjective.
\item $B\adj: Z' \rightarrow Y'$ is injective
and ${\IM}(B\adj)$  is closed in $Y'$.
\item The following holds:
\begin{equation}
\beta :=\inf_{z'\in Z'} \sup_{y\in Y} \frac{\langle B\adj z', y\rangle_{Y',Y}}{\|z'\|_{Z'} \|y\|_Y}>0.
\label{inf_sup_B_surjective}
\end{equation}
\eerom
\end{lemma}

Assume that $B$ is surjective, \ie the inf-sup condition
\eqref{inf_sup_B_surjective} holds. We now show that it is possible to
construct a right-inverse map of $B$ and to control its norm by the
inf-sup constant $\beta$.

\smallskip

\begin{lemma}[Right inverse]
\label{Lem:sur2infsup}
Assume that \eqref{inf_sup_B_surjective} holds and that $Y$ is reflexive. 
Then there is a right-inverse map $R_B:Z\to Y$ such that 
\begin{equation} \label{B:surj}
\forall z\in Z, \quad B(R_B(z))=z \quad \text{and}\quad \beta\|R_B(z)\|_Y \le \|z\|_Z.
\end{equation}
Moreover, it is possible to construct a linear map $R_B$ if $Y$ is a Hilbert space.
\end{lemma}

\begin{proof}
This result can be found in \cite[Lem.~A.42]{ErnGuermond_FEM}; for completeness,
we present a proof.
The inf-sup condition~\eqref{inf_sup_B_surjective} 
implies that $B\adj$ is injective (see Lemma~\ref{lem:surj}(ii)). Let us set $\calH:= \IM(B\adj)\subset Y'$ equipped with the norm of $Y'$.
Let $R_{B\adj} : \calH \to Z'$ be such that, for all 
$y'\in \calH$, $B\adj(R_{B\adj}(y')) = y'$. 
$B\adj$ being injective, $R_{B\adj}$ is uniquely defined and is a linear map;
notice also that $R_{B\adj}(B\adj(z'))=z'$ for all $z'\in Z'$ since 
$R_{B\adj}(B\adj(z'))-z'$ is in $\KER(B\adj)=\{0\}$. 
Moreover, the inf-sup condition~\eqref{inf_sup_B_surjective} implies that
$\|R_{B\adj}(y')\|_{Z'} \le \beta^{-1} \|y'\|_{Y'}$.
We next define the linear map $\phi : Z \to \calH'$ such that,
for all $z\in Z$, $\langle \phi(z),y'\rangle_{\calH',\calH} =
\overline{\langle R_{B\adj}(y'), z \rangle_{Z',Z}}$ for all $y'\in\calH$.
We infer that 
\[
\begin{aligned}
|\langle \phi(z),y'\rangle_{\calH',\calH}|
\leq
\|R_{B\adj}(y')\|_{Z'} \|z\|_Z
\leq&
\beta^{-1} \|y'\|_{Y'} \|z\|_Z.
\end{aligned}
\]
This means that $\phi(z)$ is
bounded on $\calH\subset Y'$ with $\|\phi(z)\|_{\calH'} \le \beta^{-1}\|z\|_Z$.
Owing to the Hahn--Banach Theorem in complex Banach spaces (see \cite[Prop.~11.23]{Brezis:11}), 
$\phi(z)$ can be extended to $Y'$ with the same norm. 
Let $\calE(\phi(z))\in Y''$ be the extension in question
with $\|\calE(\phi(z))\|_{Y''} \leq \beta^{-1} \|z\|_Z$.
Since $Y$ is reflexive, the canonical isometry $J_Y:Y\to Y''$ is a linear isomorphism.
Let us set $R_B(z):= J_Y^{-1}(\calE(\phi(z)))$; notice that
$J_Y(R_B(z))=\calE(\phi(z))$.
Then, the following holds for all $z'\in Z'$:
\begin{align*}
\langle z', B(R_B(z))\rangle_{Z',Z}
&=\langle B\adj z', R_B(z)\rangle_{Y',Y}
=\overline{\langle J_Y(R_B(z)), B\adj z' \rangle_{Y'',Y'}} 
=\overline{\langle \calE(\phi(z)), B\adj z' \rangle_{Y'',Y'}} \\
&=\langle \phi(z), B\adj z' \rangle_{\calH',\calH} 
=\langle R_{B\adj}(B\adj z'), z \rangle_{Z',Z} 
=\langle z', z \rangle_{Z',Z},
\end{align*}
showing that $B(R_B(z)) = z$. Moreover,
$\|R_B(z)\|_{Y}
=
\|J_Y(R_B(z))\|_{Y''}
=
\|\calE(\phi(z))\|_{Y''}
= \|\phi(z)\|_{\calH'}
\leq
\beta^{-1}\|z\|_Z$,
showing that \eqref{B:surj} holds. In conclusion, $R_B = J_Y^{-1} \circ \calE \circ \phi$, where the first and third maps are linear. If $Y$ is a Hilbert 
space, then
the extension map $\calE$ is also linear; it suffices to extend linear forms
on $\calH$ by zero on the orthogonal complement of $\calH$ in $Y'$.
\end{proof}

\bibliographystyle{plain} 
\bibliography{biblio}

\end{document}